\documentclass{article}
\usepackage{amsfonts}
\usepackage{amssymb}
\usepackage{graphicx}
\usepackage{amsmath}

\newtheorem{theorem}{Theorem}

\newtheorem{definition}[theorem]{Definition}

\newtheorem{lemma}[theorem]{Lemma}

\newtheorem{proposition}[theorem]{Proposition}
\newtheorem{remark}[theorem]{Remark}

\newenvironment{proof}[1][Proof]{\textbf{#1.} }{\ \rule{0.5em}{0.5em}}
\begin{document}

\title{Normal Factorization in $SL(2,\mathbb{Z})$ and the Confluence of Singular Fibers in Elliptic Fibrations}
\author{Carlos A. Cadavid, Juan D. V\'elez}
\maketitle

\begin{abstract}
In this article we obtain a result about the uniqueness of factorization in terms of conjugates of the matrix $U=\left[ \begin{array}{cc} 1 & 1 \\ 0 & 1    \end{array} \right]$, of some matrices representing the conjugacy classes of those elements of $SL(2,\mathbb{Z})$ arising as the monodromy around a singular fiber in an elliptic fibration (i.e. those matrices that appear in Kodaira's list). Namely we prove that if $M$ is a matrix in Kodaira's list, and $M=G_1\ldots G_r$ where each $G_i$ is a conjugate of $U$ in $SL(2,\mathbb{Z})$, then after applying a finite sequence of Hurwitz moves the product $G_1\ldots G_r$ can be transformed into another product of the form $H_1\ldots H_nG_{n+1}'\ldots G_r'$ where $H_1\ldots H_n$ is some fixed shortest factorization of $M$ in terms of conjugates of $U$, and $G_{n+1}'\ldots G_r'=Id_{2\times 2}$.
We use this result to obtain necessary and sufficient conditions under which a relatively minimal
elliptic fibration without multiple fibers $\phi:S\rightarrow D=\{z \in \mathbb{C}:\left|z\right|<1 \}$, admits a weak deformation into another such fibration having only one singular fiber. 
\end{abstract}

\section{Introduction}

The purpose of this article is twofold. On the one hand we begin the study of the extent to which a given element of the
mapping class group of an oriented torus (i.e. $SL(2,\mathbb{Z})$) factors uniquely as a product of right handed Dehn twists, i.e. conjugates of the matrix $$U=\left[ \begin{array}{cc} 1 & 1 \\ 0 & 1    \end{array} \right].$$
Our first main result (Theorem \ref{Main}) addresses this question, and gives an affirmative answer for those elements in $SL(2,\mathbb
{Z})$ which arise as the monodromy around a singular fiber in an elliptic fibration. As far as we know this subject has two predecesors. The first one is a well known result due to R. Livne and Moishezon
\cite{Moishezon}, which says that any factorization of the identity matrix in $SL(2,\mathbb{Z})$ in terms of $r$ conjugates of $U$ can be transformed by applying a finite sequence of Hurwitz moves, into a standard factorization $(VU)^{6s}$ where
$s\geq 0$, $r=12s$, and $$V=\left[ \begin{array}{rr} 1&0  \\-1 & 1 \end{array}  \right].$$
The second one arose in the study of branched covers of $2-$manifolds, and was initiated by Hurwitz, Clebsch and Luroth, and
more recently continued by several other authors (see \cite{KhovanskiiZdravkovska} and the references therein). These authors
study the analogous problem when one replaces $SL(2,\mathbb{Z})$ by the symmetric group $S_n$, and right
handed Dehn twists by transpositions.  For instance, Natanzon's result (see \cite{Natanzon}) claims that if $\sigma\in S_n$, and $\sigma=
\tau_1\ldots \tau_k = \tau_1'\ldots \tau_k'$ are two factorizations in terms of transpositions, such that the subgroups
$\langle \tau_1,\ldots,\tau_k \rangle$ and $\langle \tau_1',\ldots,\tau_k' \rangle$ act transitively on the
$n$ symbols, then there exists a sequence of Hurwitz moves which transforms the product $\tau_1\ldots \tau_k $ into the product $\tau_1'\ldots \tau_k'$. In particular, this implies a result (which parallels Theorem \ref{Main}) saying that
if one picks a particular shortest transitive factorization  $\mu_1\ldots \mu_{s_{\sigma}}$ of $\sigma$ in terms of 
transpositions (i.e. such that $\langle \mu_1,\ldots, \mu_{s_{\sigma}} \rangle$ acts transitively on the $n$ symbols), then any transitive factorization $\tau_1\ldots \tau_k$ of $\sigma$ in terms of transpositions, transforms after a finite sequence of Hurwitz moves into a factorization of the form $\sigma=\mu_1\ldots \mu_{s_{\sigma}}\tau_{s_{\sigma}+1}'\ldots \tau_k'$. The proof of Theorem \ref{Main} is based on the careful study of the description of $PSL(2,\mathbb{Z})$ as the direct product $\mathbb{Z}_2  \ast \mathbb{Z}_3$ developed by R. Livne (see \cite{Moishezon}).

On the other hand, we study the problem of when an elliptic fibration over a disk can be deformed into another elliptic fibration over a disk having only one singular fiber. Our result in this direction (Theorem  \ref{coalescencia2}) provides necessary and sufficient conditions under which a given relatively minimal elliptic fibration over a disk without multiple fibers can be \emph{weakly} deformed into another such fibration having only one singular fiber (Definition \ref{weakdeformation}). Weak deformation allows the passing from one deformation family to another whenever there exists a member of each family being topologically equivalent with each other. 
This result is obtained as an application of our normal factorization result. 

This type of problem was posed by Naruki in \cite{Naruki}. In \cite{Naruki} that author considers the confluence of three singular fibers $F_1$, $F_2$ and $F_3$, of types $I_a$, $I_b$ and $I_c$ in an elliptic fibration into one singular fiber $F$ after a deformation, and studies in depth the necessary condition for the existence of the confluence that if $M_1$, $M_2$ and $M_3$ are the monodromies around $F_1$, $F_2$ and $F_3$, and $M$ is the monodromy
around the singular fiber $F$ they coalesce to, then $M=M_1M_2M_3$ and $\chi(F)=\chi(F_1)+\chi(F_2)+\chi(F_3)$. He solves the algebraic problem of classifying up to the braid group action, those triples $(M_1,M_2,M_3)$ such that
each $M_i$ is a conjugate of $U^{a_i}$, $M=M_1M_2M_3$ is a matrix in Kodaira's list , and $a_1+a_2+a_3=\chi(F_M)$  (see Table \ref{primeratabla}).  The opposite problem, namely that of finding necessary and sufficient conditions under which
a singular fiber in a fibration of arbitrary fiber genus admits a deformation which splits it into several ones has been intensely studied (see for example \cite{Takamura} and the references therein). Moishezon completely solved this problem in the elliptic case (cf. Theorem \ref{morsificacion}).

\section{Basic definitions and facts}
\emph{Througout this article, $S$ will denote a complex manifold with complex dimension 2 and $D$ will denote the open unit disk $\{z\in\mathbb{C}:|z|<1\}$.}
\begin{definition}
By an \emph{elliptic fibration} we will mean a triple $(\phi,S,D)$ where $\phi:S \rightarrow D$ is a proper surjective holomorphic map with a finite number (possibly zero) of critical values $q_1,\ldots,q_k \in D$, such that the preimage of each regular value is a (compact) connected Riemann surface of genus $1$.
\end{definition}

We will say that an elliptic fibration is \emph{singular} if it has at least one singular fiber. A singular fiber $\phi^{-1}(q_i)$ is said to be of \emph{Lefschetz type} if $$C_i:=\{p\in\phi^{-1}(q_i):p \text{ is a critical point of } \phi\}$$
is finite, and for each  $p\in C_i$ there exist holomorphic charts around $p$ and $q_i$ relative to which $\phi$ takes the form $(z_1,z_2)\rightarrow z_1^2+z_2^2$. If a fiber of Lefschetz type contains exactly one critical point, it will be said to be \emph{simple}. Every fiber $\phi^{-1}(q)$ of an elliptic fibration can be regarded as an effective divisor $w_{1,q}X_{1,q}+\ldots +w_{r_q,q}X_{r_q,q}$. A (necessarily) singular fiber $\phi^{-1}(q)$ is called a \emph{multiple fiber} if $gcd(w_{1,q},\ldots,w_{r_{q},q})>1$, and it is said to be of \emph{smooth multiple type} if it is of the form $w_{1,q}X_{1,q}$ with $w_{1,q}>1$ and $X_{1,q}$ is a smooth submanifold of $S$.
An elliptic fibration is said to be \emph{relatively minimal} if no fiber contains an embedded sphere with selfintersection $-1$. \emph{All elliptic fibrations in this article will be assumed to be relatively minimal}.\\
The Euler characteristic of the domain of an elliptic fibration can be calculated using the following formula which is analogous to the Riemann-Hurwitz formula
$$\chi(S)=\sum_{i=1}^k \chi(\phi^{-1}(q_i)) $$

Next we define when two elliptic fibrations will be regarded as being (topologically) the same.
\begin{definition}
Two elliptic fibrations $(\phi_1,S_1,D)$ and $(\phi_2,S_2,D)$ are said to be \emph{topologically equivalent} if there exist orientation preserving diffeomorphisms $h:S_1 \rightarrow S_2$ and $h':D \rightarrow D$ such that $\phi_2 \circ h=h' \circ \phi_1$. In this case we write $(\phi_1,S_1,D)\sim(\phi_2,S_2,D)$ or simply $\phi_1\sim\phi_2$.
\end{definition}

\begin{definition}
By a \emph{family of elliptic fibrations} we will mean a triple $(\Phi,\mathcal{S},D\times D_{\epsilon})$ where 
$\mathcal{S}$ is a three-dimensional complex manifold, $D_{\epsilon}=\{z \in \mathbb{C}:|z|<\epsilon\}$ and $\Phi:\mathcal{S}\rightarrow D\times D_{\epsilon}$ is a surjective proper holomorphic map, such that
\begin{enumerate}
\item if for each $t\in D_{\epsilon}$, $D_t:=D \times \{t\}$, $\mathcal{S}_t:=\Phi^{-1}(D_t)$ 
and $\Phi_t:=\Phi|_{\mathcal{S}_t}:\mathcal{S}_t \rightarrow D_t$, then each $(\Phi_t,\mathcal{S}_t,D_t)$ is an elliptic
fibration;
\item the composition 
$\mathcal{S} \stackrel{\Phi}{\rightarrow} D \times D_{\epsilon} 
\stackrel{pr_2}{\rightarrow} D_{\epsilon}$ does not have critical points.
\end{enumerate}
A family of elliptic fibrations $(\Phi,\mathcal{S},D\times D_{\epsilon})$ is said to be a \emph{deformation} of a given elliptic fibration $(\phi,S,D)$, if $(\phi,S,D)$ is biholomorphically equivalent to $(\Phi_{0},\mathcal{S}_{0},D_{0})$, i.e. there exist biholomorphic maps $h:S\rightarrow \mathcal{S}_{0}$ and $h':D\rightarrow D$ such that $\Phi_{0}\circ h=h'\circ \phi$ .
\end{definition}

\begin{remark}\label{nota1}
It can be seen that if $(\Phi,\mathcal{S},D\times D_{\epsilon})$ is a family of elliptic fibrations, the (oriented) diffeomorphism type of $\mathcal{S}_t$ is independent of $t\in D_{\epsilon}$. In particular, $\chi(\mathcal{S}_t)$ is also independent of $t\in D_{\epsilon}$.
\end{remark}

\begin{definition}
Let $(\phi,S,D)$ be an elliptic fibration.  A deformation $(\Phi,\mathcal{S},D\times D_{\epsilon})$ of $(\phi,S,D)$ will be said to be a \emph{morsification of} $(\phi,S,D)$, if for each $t\neq 0$, each singular fiber of 
$\Phi_t:\mathcal{S}_t \rightarrow D_t$ is either of simple Lefschetz type or of smooth multiple type.
\end{definition}

The following fundamental result is due to Moishezon (see \cite{Moishezon}).

\begin{theorem}\label{morsificacion}
Every elliptic fibration admits a morsification. Moreover, if the elliptic fibration
does not have multiple fibers, then it admits a morsification such that none of its members 
contains a multiple fiber.
\end{theorem}
The following definition is introduced in order to state one of our main results.

\begin{definition}\label{weakdeformation}
Two elliptic fibrations 
$(\phi_1,S_1,D)$ and $(\phi_2,S_2,D)$ will be said to be \emph{weakly deformation equivalent} 
whenever there exist a finite collection of families of elliptic
fibrations $(\Phi^{1},\mathcal{S}
^{1}),\ldots,(\Phi^{k},\mathcal{S}^{k})$, and $s_i,t_i\in D_{\epsilon_i}$ 
for each $i=1,\ldots,k$, such that $\phi_1\sim\Phi _{s_1}^{1}$, 
$\Phi_{t_i}^{i}\sim\Phi_{s_{i+1}}^{i+1}$ for $i=1,\ldots,k-1$, 
and $\Phi_{t_k}^{k}\sim\phi_2.$
\end{definition}

We now turn to the combinatorial description of elliptic fibrations. Let $(\phi,S,D)$ be 
an elliptic fibration and let $q_1,\ldots,q_k$ be its critical values. Take $0<r<1$ such that the open disk
$D_r$ having center $0$ and radius $r$ contains the points $q_1,\ldots,q_k$. Let us fix a point $q_0
\in \partial D_r$. Notice that $q_0$ is a regular value.  Let us also fix an orientation preserving 
diffeomorphism $j$ between the genus 1 Riemann surface $\phi^{-1}(q_0)$ and the genus 1 Riemann surface 
$\mathbb{C}/\mathbb{Z}^2$. These choices uniquely determine an antihomomorphism 
$$\lambda_{r,q_0,j}:\pi_1(D-\{q_1,\ldots,q_k\},q_0)\rightarrow SL(2,\mathbb{Z})$$ where 
$SL(2,\mathbb{Z})$ is the group formed
by all $2\times 2$ integral matrices whose determinant is $1$. Such antihomomophism is said to be \emph{a re
presentation monodromy of} $(\phi,S,D)$.
\emph{In order to make the presentation more standard, we turn the monodromy representation into
a homomorphism by regarding $\pi_1(D-\{q_1,\ldots,q_k\},q_0)$ as the group whose
binary operation $\star$ is defined by $[\gamma_1]\star [\gamma_2]:=[\gamma_2].[\gamma_1]$, where 
 ``." denotes the usual composition of homotopy classes of paths}. 
The matrix $\lambda_{r,q_0,j}([C_r])$, where $C_r$ denotes the path $q_0 \exp(2\pi \sqrt{-1}t)$, $0\leq t \leq 1$, will be called \emph{the total monodromy of} $(\phi,S,D)$. 

\begin{remark}\label{nota2}
The conjugacy class of $\lambda_{r,q_0,j}([C_r])$ in $SL(2,\mathbb{Z})$ is independent of the choices $r$, $q_0$ and $j$, and that if $(\Phi,\mathcal{S},D\times D_{\epsilon})$ is a family of elliptic fibrations and $t_1,t_2\in D_{\epsilon}$ then the conjugacy classes of the total monodromies of $(\Phi_{t_1},\mathcal{S}_{t_1},D)$ and $(\Phi_{t_2},\mathcal{S}_{t_2},D)$ are the same.
\end{remark}

The group $\pi_1(D-\{q_1,\ldots,q_k\},q_0)$ 
is free and has rank $k$. We now describe a method for obtaining free bases for this group. The bases obtained by this
method will be called \emph{special bases}.
Pick closed disks $\overline D_1,\ldots,\overline D_k$ 
contained in $D_r$, 
centered at $q_1,\ldots,q_k$, respectively, and mutually disjoint. Pick simple paths 
$\beta_1,\ldots,\beta_k$ whose interiors are mutually disjoint and contained in 
$D_r-\cup \overline D_i$,
with $\beta_1(0)=\ldots=\beta_k(0)=q_0$ and $q_i^0:=\beta_i(1)\in \partial D_i$ for each $i=1,\ldots,k$, and
such that their initial velocity vectors $\beta_1'(0),\ldots,\beta_k'(0)$ are all nonzero and 
$0<\theta_1<\ldots<\theta_k<\pi$ where $\theta_i$ is the angle between the vectors
$\beta_i'(0)$ and $\sqrt{-1}q_0$. Let $\gamma_i$ be a path which starts at $q_0$, follows $\beta_i$ until
it reaches $q_i^0$, then traverses once and positively the circle $\partial D_i$, and finally 
comes back to $q_0$ following $\beta_i$ in the opposite direction. Then
$\{[\gamma_1],\ldots,[\gamma_k]\}$ is a basis for the free group $\pi_1(D-\{q_1,\ldots,q_k\},q_0)$. Notice that
$[C_r]=[\gamma_1]\ldots [\gamma_k]=[\gamma_k]\star\ldots \star[\gamma_1]$ and therefore the total monodromy  $\lambda([C_r])$ equals $\lambda([\gamma_k])\ldots \lambda([\gamma_1]) $.\\

The following proposition is standard. Its statement requires the concept of Hurwitz move which we define next.
\begin{definition}
Let $G$ be a group and let $g_1\ldots g_k$ be a product of elements of $G$. Another such product $g_1'\ldots g_k'$ is said to be \emph{obtained from $g_1\ldots g_k$ by applying a Hurwitz move} if for some
$1\leq i\leq k-1$, $g_j'=g_j$ for $j\notin \{i,i+1\}$, and either $g_i'=g_{i+1},\ g_{i+1}'=g_{i+1}^{-1}g_ig_{i+1}$ or $g_i'=g_ig_{i+1}g_i^{-1},\ g_{i+1}'=g_i$. We will also say that an ordered set $\{g_1',\ldots,g_k'\}$ is obtained from
another ordered set $\{g_1,\ldots,g_k\}$ \emph{by applying one Hurwitz move}, if the same relations hold between the $g_i'$'s
and the $g_i$'s.
\end{definition}

It is important to remark that the Hurwitz moves $$g_1\ldots g_ig_{i+1} g_k \rightarrow g_1\ldots g_{i+1}(g_{i+1}^{-1}g_ig_{i+1})\ldots g_k$$ and 
$$g_1\ldots g_ig_{i+1} g_k \rightarrow g_1\ldots (g_i g_{i+1}g_i^{-1})g_i\ldots g_k$$
are inverse of each other.
\begin{proposition}\label{proposition}
Let $(\phi,S,D)$ and $(\phi',S',D)$ be relatively minimal elliptic fibrations 
without multiple fibers and having the same number of singular fibers. Let 
$q_{1},\ldots,q_{k}$ (resp. $q_{1}',\ldots,q_{k}'$) be the critical values of $(\phi,S,D)$ (resp. $(\phi',S',D)$). Let
$\lambda$ (resp. $\lambda'$) be a monodromy representation for $(\phi,S,D)$ (resp. $(\phi',S',D)$). 
The following statements are equivalent 
\begin{enumerate}
\item $(\phi,S,D)\sim(\phi',S',D)$;
\item there exist an orientation preserving diffeomorphism 
$h:D\rightarrow D$ with $h(\{q_{1},\dots,q_{k}\})=\{q_{1}',\dots,q_{k}'\}$, $h(q_{0})=q_{0}'$, and a matrix $A\in SL(2,\mathbb{Z})$, such that $c_A\circ\lambda=\lambda' \circ h_*$, where $c_A$ denotes the automorphism of $SL(2,\mathbb{Z})$ defined by $c_A(B)=A^{-1}BA$, and $$h_*:\pi_1(D-\{q_{1},\ldots,q_{k}\},q_{0})\rightarrow \pi_1(D-\{q_{1}',\ldots,q_{k}'\},q_{0}')$$ is the group isomorphism
induced by $h$; 
\item there exist an isomorphism $$\psi:\pi_1(D-\{q_{1},\ldots,q_{k}\},q_{0})\rightarrow \pi_1(D-\{q_{1}',\ldots,q_{k}'\},q_{0}')$$
sending $[C_{r}]$ to $[C_{r'}]$,
and a matrix $A\in SL(2,\mathbb{Z})$, such that $c_A\circ\lambda=\lambda'\circ \psi$, where $c_A$ denotes the automorphism of $SL(2,\mathbb{Z})$ defined by $c_A(B)=A^{-1}BA$;
\item there exist special bases $\{ [\gamma_1],\ldots, [\gamma_k]  \}$ and $\{ [\gamma_1'],\ldots,[\gamma_k'] \}$ for the groups $\pi_1(D-\{q_1,\ldots,q_k\},q_0)$ and $\pi_1(D-\{q_1',\ldots,q_k'\},q_0')$, respectively, and a matrix $A\in SL(2,\mathbb{Z})$ such that the product $\lambda([\gamma_k])\ldots \lambda([\gamma_1])$ becomes
the product $\lambda'([\gamma_k'])\ldots \lambda'([\gamma_1'])$ after the application of a (finite) number of
Hurwitz moves, followed by the conjugation of all the elements in the resulting product by $A$.
\item for any pair of special bases $\{[\gamma_1],\ldots,[\gamma_k]\}$ and $\{[\gamma_1'],\ldots,[\gamma_k']\}$ for the groups $\pi_1(D-\{q_1,\ldots,q_k\},q_0)$ and $\pi_1(D-\{q_1',\ldots,q_k'\},q_0')$, respectively, there exists a matrix $A\in SL(2,\mathbb{Z})$ such that the product $\lambda([\gamma_k])\ldots\lambda([\gamma_1])$ becomes
the product $\lambda'([\gamma_k'])\ldots\lambda'([\gamma_1'])$ after the application of a (finite) number of
Hurwitz moves, followed by the conjugation of all the elements in the resulting product by $A$.
\end{enumerate}
\end{proposition}
In the rest of this section $\pi_1(D-\{q_1,\ldots,q_k\},q_0)$ (resp. $\pi_1(D-\{q_1',\ldots,q_k'\},q_0')$) will be abreviated by $\pi_1$ (resp. $\pi_1'$).

The equivalence $1\Leftrightarrow 2$ is a particular case of the result mentioned immediately after the 
statement of Theorem 2.4 of \cite{mat}. 

$2\Rightarrow 3$ is immediate, but its reciprocal is less obvious. Let $f:D\rightarrow D$ be an orientation preserving diffeomorphism such that $f(q_i')=q_i$ for $i=0,\ldots,k$.
It is enough to prove that the automorphism $f_*\circ \psi:\pi_1\rightarrow \pi_1$ (which preserves $[C_r]$) equals $h_*$ for some orientation preserving
diffeomorphism $h:D\rightarrow D$ such that $h(\{q_1,\ldots,q_k\})=\{q_1,\ldots,q_k\}$ and $h(q_0)=q_0$. 
Actually, let us see that every automorphism $\varphi$ of $\pi_1$ such that $\varphi([C_r])=[C_r]$ is induced by some orientation preserving diffeomorphism $h$ with the properties 
described in the last sentence. Let $\{[\gamma_1],\ldots,[\gamma_k]\}$ be a special basis for the group $\pi_1$, let $F(x_1,\ldots,x_k)$ be the free group in the alphabet $\{x_1,\ldots,x_k\}$ and
$\nu:F(x_1,\ldots,x_k)\rightarrow \pi_1$ the isomorphism sending $x_i$ to $[\gamma_i]$ for 
$i=1,\ldots,k$. 
Notice that $\nu(x_k\ldots x_1)=[\gamma_k]\star \ldots \star [\gamma_1] = [C_r]$. The well known fact (see \cite{Magnus}) that the group of automorphisms of $F(x_1,\ldots,x_k)$ which send the product $x_k\ldots x_1$ to itself, is generated by the elementary automorphisms $\{\phi_1,\ldots, \phi_{k-1}\}$ such that $\phi_i(x_j)=x_j$ if $j\notin \{i,i+1\}$, and $\phi_i(x_i)=x_i^{-1}x_{i+1}x_i$, 
$\phi_i(x_{i+1})=x_i$, allows us to reduce the problem to proving that each automorphism of $\pi_1$ defined as
$\varphi_i:=\nu \circ \phi_i \circ \nu^{-1}$ for $i=1,\ldots,k-1$, is induced by some orientation preserving diffeomorphism $h_i:D\rightarrow D$ with $h_i(\{q_1,\ldots,q_k\})=\{q_1,\ldots,q_k\}$ and $h_i(q_0)=q_0$. $h_i$
is explicitly constructed as a half twist performed on an appropiately chosen annulus containing the points 
$q_i$ and $q_{i+1}$.

$2 \Rightarrow 4$ is an immediate consequence of the fact that if $\{[\gamma_1],\ldots,[\gamma_k]\}$ is a special basis
for $\pi_1$ then for any orientation preserving diffeomorphism $h:D\rightarrow D$ with
$h(\{q_1,\ldots,q_k\})=\{q_1',\ldots,q_k'\}$ and $h(q_0)=q_0'$, $h_*([\gamma_1]),\ldots,h_*([\gamma_k])$ is a special basis
for $\pi_1'$.

$4\Rightarrow 5$.  Let $\{[\delta_1],\ldots,[\delta_k]\}$ (resp. $\{[\delta_1'],\ldots,[\delta_k']\}$) be special bases for $\pi_1$ (resp. $\pi_1'$). We have that $[\delta_k] \star \ldots \star [\delta_1]=
[\gamma_k]\star \ldots \star [\gamma_1]$ and $[\delta_k'] \star \ldots \star [\delta_1']=
[\gamma_k'] \star \ldots \star [\gamma_1']$. The well known fact from \cite{Magnus} invoked above is equivalent to the fact that
if $y_1,\ldots,y_k$ and $z_1,\ldots,z_k$ are free bases for $F(x_1,\ldots,x_k)$, such that $y_k\ldots y_1=z_k\ldots z_1$,
then there exists a finite sequence of Hurwitz moves that transforms the product $y_k\ldots y_1$ into the product
$z_k \ldots z_1$. Applied to our situation this gives the existence of a finite sequence of Hurwitz moves which transforms
the product $[\delta_k] \star \ldots \star [\delta_1]$
into the product $[\gamma_k]\star \ldots \star [\gamma_1]$, and another finite sequence of 
Hurwitz moves transforming the product $[\gamma_k'] \star \ldots \star [\gamma_1']$ into the product $[\delta_k'] \star \ldots \star [\delta_1']$. Combining
this with the existence of a sequence of Hurwitz moves and a conjugation transforming the product
$\lambda([\gamma_k])\ldots \lambda([\gamma_1])$ into the product $\lambda'([\gamma_k'])\ldots \lambda'([\gamma_1'])$
allows us to conclude that there exists a finite sequence of Hurwitz moves and a conjugation transforming
the product $\lambda([\delta_k])\ldots \lambda([\delta_1])$ into the product $\lambda'([\delta_k'])\ldots \lambda'([\delta_1'])$.

$5 \Rightarrow 3$. Let $\{[\gamma_1],\ldots,[\gamma_k]\}$ (resp. $\{[\gamma_1'],\ldots,[\gamma_k']\}$) be a special basis for $\pi_1$ (resp. $\pi_1'$). Then the product 
$\lambda([\gamma_k])\ldots \lambda([\gamma_1])$ can be transformed to the product $\lambda'([\gamma_k'])\ldots \lambda'([\gamma_1'])$ by applying a sequence $\mu_1,\ldots, \mu_l$ of Hurwitz moves, followed by the conjugation
of all the elements
in the resulting product by a matrix $A\in SL(2,\mathbb{Z})$. Let $\{[\gamma_1''],\ldots,  [\gamma_k'']\}$ be the 
special basis for $\pi_1$ obtained by applying the sequence $\mu_l^{-1},\ldots,
\mu_1^{-1}$ of Hurwitz moves to $\{[\gamma_1'],\ldots,[\gamma_k']\}$. Let $\psi:\pi\rightarrow \pi'$ be the isomorphism 
determined $\psi([\gamma_i])=[\gamma_i'']$ for each $i=1,\ldots,k$. It can easily verified that $c_A\circ \lambda=\lambda' \circ \psi$.

\section{Kodaira's list}

Let $(\phi,S,D)$ be an elliptic fibration and let $q\in D$. It is a well known fact that the fiber $\phi^{-1}(q)$ is
a triangulable topological space. In \cite{Kodaira} Kodaira studied the problem of classifying fibers in elliptic fibrations under the following  equivalence relation which takes into account not only the topological structure of 
the fiber but also the structure of the map $\phi$ in a regular neighborhood of it.

\begin{definition}
Let $(\phi,S,D)$ and $(\phi',S',D)$ be elliptic fibrations and let $q,q'\in D$. Let $\sum m_iX_i$ and $\sum n_jY_j$ be the effective divisors associated to $\phi^{-1}(q)$ and $(\phi')^{-1}(q')$, respectively. The fibers $\phi^{-1}(q)$ and $(\phi')^{-1}(q')$ are said to be \emph{of the same type} if there is a homeomorphism $f:\phi^{-1}(q)\rightarrow (\phi')^{-1}(q')$, so that the induced map $f_*:H_2(\phi^{-1}(q);\mathbb{Z})\rightarrow H_2((\phi')^{-1}(q');\mathbb{Z})$ sends the class $\sum m_i[X_i]$ to the class $\sum n_j[Y_j]$. 
\end{definition}

We will rely heavily on the following classical result due to Kodaira (see \cite{Kodaira}). 

\begin{theorem}\label{Kodaira}
Let $(\phi,S,D)$ be a relatively minimal elliptic fibration and let $q_i$ be a critical value of $\phi$. Then 
\begin{enumerate}
\item the fiber $\phi^{-1}(q_i)$ is of the same type of one and only one of the following pairs:\\
$wI_0$: $wX_0,\ w>1$ where $X_0$ is a non-singular elliptic curve.\\
$wI_1$: $wX_0,\ w\geq1$ where $X_0$ is a rational curve with an ordinary double point.\\
$wI_2$: $wX_0+wX_1,\ w\geq1$ where $X_0$ and $X_1$ are non-singular rational curves with intersection $X_0 \cdot X_1=p_1+p_2$.\\
$II$: $1X_0$ where $X_0$ is a rational curve with one cusp.\\
$III$: $X_0+X_1$ where $X_0$ and $X_1$ are non-singular rational curves with $X_0\cdot X_1=2p$.\\
$IV$: $X_0+X_1+X_2$, where $X_0,X_1,X_2$ are non-singular rational curves and $X_0 \cdot X_1=X_1 \cdot X_2=X_2\cdot X_0=p$.\\
The rest of the types are denoted by $wI_b,\ b\geq 3$, $I_b^*$, $II^*$, $III^*$, $IV^*$ and are composed of non-singular rational curves $X_0,X_1,\ldots,X_s,\ldots$ such that $X_s \cdot X_t \leq 1$ (i.e. $X_s$ and $X_t$ have at most one simple intersection point) for $s<t$ and $X_r \cap X_s \cap X_t$ is empty for $r<s<t$. These types are therefore described completely by showing all pairs $X_s,\ X_t$ with $X_s\cdot X_t=1$ together with  $\sum w_iX_i$.\\
$wI_b$: $wX_0+wX_1+\ldots+wX_{b-1}$, $w=1,2,3,\ldots$, $b=3,4,5,\ldots$, $X_0 \cdot X_1=X_1\cdot X_2=\ldots=X_s\cdot X_{s+1}=\ldots =X_{b-2}\cdot X_{b-1}=X_{b-1}\cdot X_0=1$.\\
$I_b^*$: $X_0+X_1+X_2+X_3+2X_4+\ldots+2X_{4+b}$ where $b\geq 0$, and $X_0 \cdot X_4=X_1 \cdot X_4= X_2 \cdot X_{4+b}=X_3 \cdot X_{4+b}=X_4 \cdot X_5=X_5 \cdot X_6=\ldots=X_{3+b}\cdot X_{4+b}=1$.\\
$II^*$: $X_0+2X_1+3X_2+4X_3+5X_4+6X_5+4X_6+3X_7+2X_8$, where $X_0 \cdot X_1=X_1 \cdot X_2= X_2 \cdot X_3=X_3 \cdot X_4=X_4 \cdot X_5=X_5 \cdot X_7=X_5\cdot X_6=X_6\cdot X_8=1$.\\
$III^*$: $X_0+2X_1+3X_2+4X_3+3X_4+2X_5+2X_6+X_7$, where $X_0 \cdot X_1=X_1 \cdot X_2= X_2 \cdot X_3=X_3 \cdot X_5=X_3 \cdot X_4=X_4 \cdot X_6=X_6\cdot X_7=1$.\\
$IV^*$: $X_0+2X_1+3X_2+2X_3+2X_4+X_5+X_6$, where $X_0 \cdot X_1=X_1 \cdot X_2= X_2 \cdot X_3=X_2 \cdot X_4=X_3 \cdot X_5=X_4 \cdot X_6=1$.\\
\item the conjugacy class of $\lambda([\gamma_i])$, where $[\gamma_i]$ is the $i^{th}$ term in any special basis for $\pi_1(D-\{q_1,\ldots,q_k\},q_0)$, depends only on the type of the fiber $\phi^{-1}(q_i)$.\\
\item for each type $T$ above there exists a relatively minimal elliptic fibration $(\phi_T,S_T,D)$ with $F_T:=\phi_T^{-1}(0)$ as its unique singular fiber, and having type $T$.
\end{enumerate}
\end{theorem}

 The following table contains for each type $T$, a matrix representative $M_T$ of the conjugacy class of the total monodromy of $(\phi_T,S_T,D)$, the Euler characteristic of $S_T$ (which is the same as the Euler characteristic of $F_T$) (see \cite{VandeVen}), and 
a particular factorization of $M_T$ in $SL(2,\mathbb{Z})$ which will play a central role in next section.
\begin{equation}\label{primeratabla}
\begin{array}{cccccc}
T & \ \ &  M_T & \ \ & \chi(S_T) & m.n.f. \\
    \ \ & \ \ & \ \ & \ \ & \ \ \\ 
wI_n \ (w\geq 1,n\geq 0) &\ \ &   \left[\begin{array}{rr} 1 & n \\ 0 & 1 \end{array} \right] & \ \ & n  &  U^{n}\\
        \ \ & \ \ & \ \ & \ \ & \ \ \\ 
II & \ \ & \left[\begin{array}{rr} 1 & 1 \\ -1 & 0 \end{array} \right]& \ \ & 2\ \ & VU\\
      & \ \  & \\
III  & \ \ &\left[\begin{array}{rr} 0 & 1 \\ -1 & 0 \end{array} \right]& \ \ & 3\ \ & VUV \\
          \ \ & \ \ & \ \ & \ \ & \ \ \\ 
IV  & \ \ & \left[\begin{array}{rr} 0 & 1 \\ -1 & -1 \end{array} \right]& \ \ & 4\ \ & (VU)^2\\
          \ \ & \ \ & \ \ & \ \ & \ \ \\ 
I_n^* \ (n\geq 0) & \ \ &\left[\begin{array}{rr} -1 & -n \\ 0 & -1 \end{array} \right] & \ \ & n+6\ \ & U^{n}(VU)^{3} \ (=-U^{n}) \\
          \ \ & \ \ & \ \ & \ \ & \ \ \\ 
II^* & \ \ &\left[\begin{array}{rr} 0 & -1 \\ 1 & 1 \end{array} \right] & \ \ & 10\ \ & VU(VU)^3 \ (=-VU) \\
          \ \ & \ \ & \ \ & \ \ & \ \ \\ 
III^* & \ \ &\left[\begin{array}{rr} 0 & -1 \\ 1 & 0 \end{array} \right] & \ \ & 9\ \ & VUV(VU)^3 \ (=-VUV) \\
        \ \ & \ \ & \ \ & \ \ & \ \ \\ 
IV^* &  \ \ &\left[\begin{array}{rr} -1 & -1 \\ 1 & 0 \end{array} \right]& \ \ & 8\ \ & (VU)^2(VU)^3 \ (=-(VU)^2) \\
\end{array}
\end{equation}

\section{Factorization of Kodaira's matrices in terms of conjugates of $U$}

In this section we recall some basic facts about the group $SL(2,\mathbb{Z})$, formed by all  $2\times 2$ matrices with integral entries and determinant
$1$, and about the modular group $PSL(2,\mathbb{Z)}$\ defined as
the quotient $SL(2,\mathbb{Z)}/\{\pm Id_{2\times 2}\}$, and prove some uniqueness results (Theorems \ref{Main} and \ref{fuerte}) about the factorization in terms of conjugates of 
the matrix 
$$U=\left[ 
\begin{array}{cc}
1 & 1 \\ 
0 & 1%
\end{array}
\right],$$
of the matrices appearing in Kodaira's list (second column of Table \ref{primeratabla}).

\subsection{Study of $PSL(2,\mathbb{Z})$ as $\left\langle w,b\left\vert \ w^{2}=b^{3}=1\right.
\right\rangle$}
Although a significant part of the material in this section can be found in references \cite{Moishezon},
\cite{Friedman-Morgan}, \cite{Matsumoto}, for the sake of
completeness we have included complete proofs of those results that are more specialized.

In what follows we will refer to particular elements (classes) in $PSL(2,
\mathbb{Z})$ by specifying one of its representatives. \emph{We will use
capital letters for the elements of} $SL(2,Z)$ \emph{and the corresponding 
lower case letters for their images in $PSL(2,\mathbb{Z})$.} For example, since 
$U=\left[ 
\begin{array}{cc}
1 & 1 \\ 
0 & 1
\end{array}
\right]$
then $u$ denotes the class $\pm U$. 

It is a well known fact that the modular group is isomorphic to the free
product $\mathbb{Z}_{2}\ast \mathbb{Z}_{3}$ via an isomorphism taking a generator of
$\mathbb{Z}_{2}$ to 
$w=\left[ 
\begin{array}{cc}
0 & 1 \\ 
-1 & 0
\end{array}
\right],$ and a generator of $\mathbb{Z}_{3}$ to $b=wu$. Hence, 
\begin{equation*}
G=PSL(2,\mathbb{Z})\cong \left\langle w,b\left\vert \ w^{2}=b^{3}=1\right.
\right\rangle .
\end{equation*}
From this we see that the abelianization of $G$ is $\mathbb{Z}_{2}\times 
\mathbb{Z}_{3}$ with any conjugate of $u$ being sent to $1$. Consequently, the abelianization of 
$SL(2,\mathbb{Z})$ is $\mathbb{Z}_{12}$, with any conjugate of the matrix 
$U$ being sent to $1$.

It also follows that each element $a\neq id_{2\times 2}$ in this group can be written uniquely
as a product $a=t_{k}\cdots t_{1}$, where each $t_{i}$ is either $w,b,$ or 
$b^{2}$ and no consecutive pair $t_{i+1}t_{i}$ is formed either by two powers
of $b$ or two copies of $w$. We call the product $t_{k}\cdots t_{1}$ the 
\emph{reduced expression}\textit{\ of }$a,$ and $k$ the \emph{length} of $a,$
which we will denote it by $l(a).$ Let $c=t_{1}^{\prime }\cdots t_{l}^{\prime
} $ be the reduced expression of an element $c\neq id_{2\times 2}$. If exactly the first $m\geq1$
terms of $c$ cancel with those of $a$, i.e. $t_{i}^{\prime }=t_{i}^{-1}$,
for $1\leq i\leq m$, and if $m< \min (k,l),$ then $ac=t_{k}\cdots
t_{m+1}t_{m+1}^{\prime }\cdots t_{l}^{\prime }$ and $t_{m+1}t_{m+1}^{\prime }$ has
to be equal to a non trivial power of $b.$ This is because if $t_{m+1}$ were
not a power of $b$ then it would have to be $w$ and therefore $t_{m}$
would be a first or second power of $b,$ and so would be $t_{m}^{\prime}.$
Hence, $t_{m+1}^{\prime }$ would also have to be $w$ but in this case there
would be $m+1$ instead of $m$ cancellations at the juncture of $a$ and $c$.
Thus, $t_{m+1}$ and $t_{m+1}^{\prime}$ are both powers of $b$ and since there
are exactly $m$ cancellations their product must be non trivial. Thus, the
reduced expression for $ac$ is of the form
\begin{equation}
ac=t_{k}\cdots t_{m+2}b^{e}t_{m+2}^{\prime }\cdots t_{l}^{\prime }\text{, }%
e=1\text{ or }2\text{, \ if }m < \min (k,l).  \label{F0}
\end{equation}

Let $s_{1}$ denote the element $bwb$. The shortest conjugates of $s_{1}$ in 
$G$ are precisely $s_{0}=b^{2}(bwb)b=wb^{2}$ with length 2, $s_{2}=b(bwb)b^{2}=b^{2}w$ 
with length 2, and $s_{1}$ itself with length 3. It
can easily be seen that any other conjugate $g$ of $s_{1}$ has length greater than $3$, and
that its reduced expression is of the form $q^{-1}s_{1}q,$ where $q$ is a reduced word that
begins with $w$ (see \cite{Friedman-Morgan}), therefore $l(g)=2l(q)+3$. A conjugate $g$ of $s_{1}$ will be called \emph{short} if $g\in \{s_{0},s_{1},s_{2}\},$ and it will be called \emph{long} otherwise.

\label{automorfismo}\emph{Let} $c_{b}(a)=b^{-1}ab$ \emph{denote conjugation
by} $b$. \emph{This is an automorphism of} $G$ \emph{that sends} $u=wb$
\emph{to} $b^{2}wb^{2}$. \emph{The map} $\varphi :\mathbb{Z}_{2}\ast
\mathbb{Z}_{3}\rightarrow \mathbb{Z}_{2}\ \ast \mathbb{Z}_{3}$ \emph{defined by sending} $w$ \emph{to itself, and} $b$ \emph{to} $b^{2}$, \emph{that is}, $\varphi
=Id\ast \psi$, \emph{where} $\psi$ \emph{is the automorphism of} $\mathbb{Z}_{3}$
\emph{\ that sends }$b$\emph{\ to }$b^{2}$\emph{, is an automorphism that
maps }$b^{2}wb^{2}$\emph{\ to }$s_{1}.$\emph{\ }Hence the composite $\varphi \circ c_b$ of these
two automorphisms is an automorphism $\rho $ that sends $u$ and $v$ into $%
s_{1}$ and $s_{0},$ respectively, and takes conjugates of $u$ into
conjugates of $s_{1}$.\\

The following notion is the key ingredient for understanding the reduced
expression of a product of conjugates of $s_{1}.$

\begin{definition}
We will say that two conjugates $g$ and $h$ of $s_{1}$ \emph{join well} if 
$$l(gh)\geq \max (l(g),l(h)).$$
\end{definition}

In \cite{Friedman-Morgan} (Lemma 4.10) the following result is proved.

\begin{lemma}
\label{largo}Suppose that $g=t_{k}\cdots t_{1}$ and $h=t_{1}^{\prime }\cdots
t_{l}^{\prime }$ are the reduced expressions of two conjugates of $s_{1}$
that join well. When $gh$ is calculated either:

\begin{enumerate}
\item no cancellation occurs, and in this case $t_{k}\cdots
t_{1}t_{1}^{\prime }\cdots t_{l}^{\prime }$ is the reduced expression of $gh$,
 or

\item exactly the first $m\geq 1$ terms of $g$ and $h$ cancel out, in which case
\begin{equation} \label{desigualdad}
m < \min (k,l). 
\end{equation} 
Moreover, if $g$ is short or $h$ is
short, then both are short and they are $s_{2}$ and $s_{0}$, respectively. 
If both are
long with reduced expressions of the form $g=q_{1}^{-1}s_{1}q_{1}$ and $
h=q_{2}^{-1}s_{1}q_{2}$, hence with lenghts $2l(q_{i})+3$, then the
reduced expression of $gh$ is of the form 
\begin{equation*}
gh=t_{k}\cdots t_{m+2}b^{e}t_{m+2}^{\prime }\cdots t_{l}^{\prime }\text{, }%
e=1 \text{or} 2,
\end{equation*}
and the inequality (\ref{desigualdad}) can be improved to 
$$m < \min
((k-1)/2,(l-1)/2)$$
which implies that $m \leq \min (l(q_{1}),l(q_{2})).$
\end{enumerate}
\end{lemma}

Now suppose that $g$ and $h$ are two conjugates of $s_1$ that do not join well, and
are not both short. The next lemma shows that in this case there exist 
$g^{\prime }$ and $h^{\prime }$ conjugates of $s_{1}$ such that $gh=g^{\prime
}h^{\prime }$ and $l(g^{\prime })+l(h^{\prime})<l(g)+l(h)$ (\cite
{Friedman-Morgan}, Proposition 4.15).

\begin{lemma}
\label{fridman}Suppose that $g$ and $h$ are conjugates of $s_{1}$ which
satisfy the inequality $l(gh)$ $<\max (l(g),l(h)),$ and assume that at least
one of them is long. Then $l(g) \neq l(h).$ If $l(g)< l(h)$, the
elements $g^{\prime }=ghg^{-1}$, $h^{\prime }=g$ are conjugates of $s_{1}$ and
satisfy:

\begin{enumerate}
\item $gh=g^{\prime }h^{\prime}$, and

\item $l(g^{\prime })+l(h^{\prime })<l(g)+l(h).$
\end{enumerate}

If instead, $l(h)<l(g),$ then the same conclusion holds taking $g^{\prime
}=h,$ and $h^{\prime }=h^{-1}gh.$
\end{lemma}

Notice that in the previous proof, the pair $(g',h')$ is obtained from the pair $(g,h)$ 
by performing one Hurwitz move.

Using the previous lemma we can prove that a product $g_{1}\cdots g_{r}$ of
conjugates of $s_{1}$ can always be transformed by applying a finite number of Hurwitz moves into a product $g_{1}^{\prime
}\cdots g_{r}^{\prime}$ of conjugates of $s_1$ in which each pair of consecutive
terms joins well. Notice that if $g_1'\ldots g_s'$ is obtained from $g_1\ldots g_r$ by applying a finite number of Hurwitz moves, then $s=r$, $g_1'\ldots g_s'=g_1\ldots g_r$ and $\{C(g_1'),\ldots,C(g_r')\}=\{C(g_1),\ldots,C(g_r)\}$ where $C(g)$ denotes the conjugacy
class of $g$.

\begin{proposition}
\label{Key}Let $g_{1}\ldots g_{r}$ be a product of $r$ conjugates of 
$s_{1}$. Then after a finite number of Hurwitz moves one can obtain a new
product $g_{1}^{\prime }\ldots g_{r}^{\prime }$ of conjugates of $s_1$,  
such that
either they are all short, or any pair of consecutive factors 
$g_{i}^{\prime}g_{i+1}^{\prime}$ join well.
\end{proposition}

Before proving this proposition we need to know how to handle pairs of consecutive short conjugates of $s_1$ that do not join well.

\begin{proposition}
\label{short} Let $p=s_{i_{1}}s_{i_{2}}\cdots s_{i_{l}}$ with $l\geq2$ be a product of
short conjugates of $s_1$, where there is at least one pair of consecutive terms
that do not join well. Then after a finite number of Hurwitz moves, $p$ can
be written as a product $s_{j_{1}}s_{j_{2}}\cdots s_{j_{l}}$ of short 
conjugates of $s_1$ (with the
same number of terms) where $s_{j_{1}}$ can be chosen
arbitrarily from the set $\{s_{0},s_{1},s_{2}\}$. In a similar way, $
s_{i_{1}}s_{i_{2}}\cdots s_{i_{l}}$ can be transformed by applying a finite number of Hurwitz moves into another product of short
conjugates of $s_1$ with the same number of terms, where the last conjugate can be
chosen arbitrarily.
\end{proposition}

\begin{proof}
We use induction on $l$. For $l=2$ a direct computation shows that the pairs 
$s_{2}s_{0}$, $s_{0}s_{1}$, and $s_{1}s_{2}$ are the only ones that do not join
well, and that each product equals $b$. From this the claim follows by
noticing that each product can be changed into any other by a Hurwitz move:
$s_{2}s_{0}=s_{0}(s_{0}^{-1}s_{2}s_{0})=s_0s_1$, 
$s_{0}s_{1}=s_{1}(s_{1}^{-1}s_{0}s_{1})=s_{1}s_{2},$ and $
s_{1}s_{2}=s_{2}(s_{2}^{-1}s_{1}s_{2})=s_{2}s_{0}.$\\

Now let $l>2$. If the first pair does not join well, the same argument as
before could be applied. Hence, we may assume that there is a consecutive pair in the product $s_{i_{2}}\cdots s_{i_{l}}$ which does
not join well.  Then, by induction
we can change this product by a new product $s_{j_{2}}\cdots s_{j_{l}},$
where $s_{j_{2}}$ can be made to be any short conjugate. Consequently, if $
s_{i_{1}}$ is $s_{0}$ (resp. $s_{1},s_{2}$) then we may choose $
s_{j_{2}}$ to be $s_{1}$ (resp. $s_{2},s_{0}$) so that the first
pair does not join well and therefore can be changed again by a pair whose
first term can be chosen arbitrarily. \\

The proof of the second part is analogous.
\end{proof}\\

\begin{proof}[Proof of Proposition \protect\ref{Key}]
Among all products $g_{1}^{\prime }\ldots g_{r}^{\prime
}$ obtained by Hurwitz moves from the product $g_{1}\ldots g_{r}$
we may choose one such that the sum $\sum\nolimits_{i=1}^{r}l(g_{i}^{\prime
})$ is as small as possible. If all $g_{i}^{\prime }$ are short we are done.
If not, any $g_{i}^{\prime }$ which is long has to join well with any
term (if any) before or after it, for otherwise, by Lemma \ref{fridman}, the
corresponding pair could be transformed by a Hurwitz move into another one making the sum 
$\sum\nolimits_{i=1}^{r}l(g_{i}^{\prime })$ smaller.

On the other hand, let $s_{i_{1}}\cdots s_{i_{l}}$ be a product of
consecutive short conjugates that appears in $g_{1}^{\prime }\cdots
g_{r}^{\prime }$. If this product precedes a long conjugate $g_{k}^{\prime}$,
i.e., if $s_{i_{1}}\cdots s_{i_{l}}g_{k}^{\prime }$ is a segment of the
product $g_{1}^{\prime }\cdots g_{r}^{\prime }$, then, by the previous lemma,
either any pair of consecutive elements of $s_{i_{1}}\cdots s_{i_{l}}$ joins well or
we can transform this product via Hurwitz moves into $s_{j_{1}}s_{j_{2}}\cdots s_{j_{l}},$ where $
s_{j_{l}}$ can be chosen arbitrarily. If the reduced expression of $
g_{k}^{\prime }$ is of the form $wb^{e}t_{3}\cdots t_{m},$ $e=1$ or $2$, we may
choose $s_{j_{l}}=s_{2}$ so that $s_{j_{l}}$ and $g_{k}^{\prime }$ do not
join well. In this situation Lemma \ref{fridman} guarantees that by applying one Hurwitz
move we would obtain a new product whose length sum is smaller than that of $g_1'\ldots g_r'$. But this is a contradiction.  Similarly, if the
reduced expression of $g_{k}^{\prime }$ is of the form $b^{e}wt_{3}\cdots
t_{m}$, with $e=1$ or $2$, then if $e=1$ (resp. $e=2$), we could choose $s_{j_{l}}$ to be $s_{2}$ (resp. $s_1$) so that $
s_{j_{l}}$ and $g_{k}^{\prime }$ do not join well. As before, this leads to a contradiction. We conclude that any pair of consecutive
elements of $s_{i_{1}}\cdots s_{i_{l}}$ must join well. The argument is
essentially the same in case $g_{k}^{\prime }s_{i_{1}}\cdots s_{i_{l}}$ is a segment
of the product $g_{1}^{\prime }\cdots g_{r}^{\prime}$. Thus, we see that if
a product $g_{1}^{\prime }\ldots g_{r}^{\prime }$ obtained from $g_1\ldots g_r$ via Hurwitz moves and minimizing the sum 
$\sum\nolimits_{i=1}^{r}l(g_{i}^{\prime })$ among such products, contains at least one long
conjugate then all consecutive pairs in it must join well. This proves the
proposition.
\end{proof}\\

Let us define the \emph{left end} of a conjugate $g$ of $s_1$, denoted by $\mathrm{left}(g)$, as
follows: If $g$ is long of the form $g=q^{-1}s_{1}q$, define $\mathrm{left}%
(g)=q^{-1}s_{1}.$ If $g$ is $s_{0}=wb^{2},s_{1}=bwb,$ or $s_{2}=b^{2}w,$ we
define its left end as $w,b,b^{2},$ respectively.

\begin{lemma}
\label{left}If in a product $p=g_{1}\cdots g_{r}$ of conjugates of $s_1$ all pairs of
consecutive factors join well, then the reduced expression of 
 $p$ is of the form $\mathrm{left}(g_{1})t_{1}\cdots t_{l}$, where each 
 $t_{i}$ is one of $b$, $b^{2}$ or $w$.
\end{lemma}

\begin{proof}
We prove this by induction on $r$, the assertion being trivial for $r=1$. We
distinguish several cases.

\begin{enumerate}
\item $g_{1}=s_{0}$. Since $g_{1}$ and $g_{2}$ join well, Lemma \ref{largo} implies that
either $g_2$ is short or no cancellation occurs when the product $g_1g_2$ is calculated. If $g_{2}$ is short, then it should be equal to $s_{0}$ or to $s_{2}$. If $g_2=s_0$, by the induction hypothesis, we must have that the reduced expression
for $g_{2}\cdots g_{r}$ is of the form $wt_{1}\cdots t_{k}$, and
consequently $g_{1}\cdots g_{r}=wb^{2}wt_{1}\cdots t_{k}=\mathrm{left}
(s_{0})t_{1}^{\prime }\cdots t_{l}^{\prime }$. On the
other hand, if $g_{2}=s_{2}$ then $g_{2}\cdots
g_{r}=b^{2}t_{1}\cdots t_{k}$ and the result also holds, since $g_{1}\cdots g_{r}=wbt_{1}\cdots t_{k}=\mathrm{left}(s_{0})t_{1}^{\prime
}\cdots t_{l}^{\prime }.$
It only rests to consider the case in which no cancellation occurs when $g_1g_2$ is calculated.  By the induction hypothesis, we know that the reduced expression of $g_2\ldots g_r$ has the form $\text{left}(g_2)t_1\ldots t_k$. On the other hand, when $g_1\text{left}(g_2)$ is calculated no cancellation occurs. We conclude that the reduced expression of $g_1\ldots g_r$ is of the form $\text{left}(g_1)t_1'\ldots t_l'$, and the results holds.

\item If $g_{1}=s_{2}$ or $g_{1}=s_1$. In these cases the argument is exactly the same as in the previous case.

\item $g_{1}$ is long. By Lemma \ref{largo}, either $g_2$ is short and no cancellation occurs when $g_1g_2$ is calculated, or $g_2$ is long.  In the first case, by the induction hypothesis, $g_2\ldots g_r=\text{left}(g_2)t_1\ldots t_k$. On the other hand, since no cancellation occurs when $g_1g_2$ is calculated we have that no cancellation occurs when $g_1\text{left}(g_2)$  is calculated.  We conclude that the reduced expression of $g_1\ldots g_r$ is of the form $\text{left}(g_1)t_1'\ldots t_l'$ and the result holds.
Let us assume now that $g_{2}$ is long. If $g_{1}=q_{1}^{-1}s_{1}q_{1}$
and $g_{2}=q_{2}^{-1}s_{1}q_{2}$ then, by Lemma \ref{largo}, either no
cancellation occurs, or the number of terms that cancel out in the product $
g_{1}g_{2}$ is $\leq \min (l(q_{1}),l(q_{2}))$. By induction $g_{2}\cdots
g_{r}=q_{2}^{-1}s_{1}t_{1}\cdots t_{k},$ and in either case the reduced
expression of $g_{1}\cdots g_{r}$ starts with $q_{1}^{-1}s_{1}.$
\end{enumerate}
\end{proof}

\subsection{Uniqueness of factorization results}

Let us set 
\begin{equation*}
W=\left[ 
\begin{array}{rc}
0 & 1 \\ 
-1 & 0
\end{array}
\right] ,
\end{equation*}
and 
\begin{equation*}
V=W^{-1}UW=-WUW=\left[
\begin{array}{rr}
1 & 0 \\ 
-1 & 1
\end{array}
\right].
\end{equation*}

The last column of Table \ref{primeratabla} contains a particular factorization of the corresponding monodromy matrix in term of conjugates of $U$ ($V$ is a conjugate of $U$). This
factorization will be called the \emph{minimal normal factorization} (which we will abreviate as m.n.f.) of the corresponding matrix. In this section we intend to prove the
following theorem.

\begin{theorem}\label{Main}
Let $M$ be one of the matrices in Table \ref{primeratabla}. If 
$M=G_{1}\cdots G_{r}$ is a factorization of $M$ in terms of conjugates of $U$ in $SL(2,\mathbb{Z})$,
then $r$
is greater than or equal to the number of factors in the m.n.f. of $M$. 
Moreover, if $n$ is such number then after a finite
number of Hurwitz moves the product $G_1\ldots G_r$ transforms 
into a product $C_{1}\cdots C_{n}D_{n+1}\cdots D_{r}$ where

\begin{itemize}
\item in cases $wI_n-IV$, $C_{1}\cdots C_{n}$ is the m.n.f. of 
$M$ and $D_{n+1}\cdots D_{r}$ is equal to the identity matrix 
$Id_{2\times 2},$ and

\item in cases $I^{\ast}_n-IV^{\ast}$, $C_{1}\cdots C_{n}$ is the m.n.f. of $-M$ and $D_{n+1}\cdots D_{r}$ is equal to 
$-Id_{2\times 2}.$
\end{itemize}
\end{theorem}

For instance, for 
\begin{equation*}
M=\left[ 
\begin{array}{rr}
0 & -1 \\ 
1 & 0
\end{array}
\right] \text{, \ }(\text{case }III^{\ast })
\end{equation*}
any factorization $M=G_{1}\cdots G_{r}$ can be transformed using Hurwitz moves into 
$M=VUVD_{4}\cdots D_{r}$ where $D_{4}\cdots D_{r}=-Id_{2\times
2}$. By a well known theorem of Moishezon \cite{Moishezon} we also know
that any factorization of the identity in terms of conjugates of $U$, can be transformed using Hurwitz moves into a product of the form $(VU)^{6s}$ with $s\geq 0$,  from which we can strengthen the theorem above as follows.

\begin{theorem}
\label{fuerte}Let $M$ be a matrix that corresponds to the monodromy of a
singular fiber in an elliptic fibration. If $M=G_{1}\cdots G_{r}$ is a
factorization of $M$ in terms of conjugates of $U$, then $r$ is greater than or equal to
the number of factors in the m.n.f. of $M$. Moreover, if $n$
is such number, then after a finite number of Hurwitz moves the product $G_1\ldots G_r$ becomes the product $C_{1}\cdots C_{n}(VU)^{6s}$, 
where 
$C_{1}\cdots C_{n}$ is the m.n.f. of $M$ and $s=(r-n)/12$,
in cases $wI_n-IV$, and into $C_{1}\cdots C_{n}(VU)^{6s+3}$, where 
$C_{1}\cdots C_{n}$ is the m.n.f. of $-M$ and $s=(r-n-6)/12$
, in cases $I^{\ast}_n-IV^{\ast}.$
\end{theorem}

We now present the proof of Theorem \ref{Main}.\\

\begin{proof} 
We deal with cases $wI_n$-$IV$ first, and once we have
established these, a rather trivial argument takes care of the
remaining cases $I^{\ast }_n$-$IV^{\ast }$. In what follows,
$\pi:SL(2,\mathbb{Z})\rightarrow PSL(2,\mathbb{Z})$ will be the
canonical homomorphism.

\textbf{Claim}: \emph{Let $M$ be one of the matrices in cases $wI_n-IV$. Suppose that $m:=\pi(M)=g_1\ldots g_r$ is a factorization in terms of conjugates of $\pi(U)=u$. Then after a finite number of Hurwitz moves the product $
g_{1}\cdots g_{r}$ transforms into a new one of the form $
(g_{1}^{\prime }\cdots g_{n}^{\prime })(g_{n+1}^{\prime }\cdots
g_{r}^{\prime })$ where if $G_1'\ldots G_n'$ is the m.n.f. of $M$, then $g_i'=\pi(G_i')$ for $i=1,\ldots,n$, and 
$g_{n+1}^{\prime }\cdots g_{r}^{\prime }=\pi(Id_{2\times 2})$}. 

Assuming this claim we can prove cases $wI_n-IV$ of the theorem as follows.
A product $H_1\ldots H_s$ in $SL(2,\mathbb{Z})$ will be said to be a \emph{lift} of a product $k_1\ldots k_s$
in $PSL(2,\mathbb{Z})$ 
if $\pi(H_i)=k_i$ for each $i=1,\ldots,s$. It can be immediatly verified that 
if a product $H_1\ldots H_s$ is a lift of a product $k_1\ldots k_s$, then the product
$H_1'\ldots H_s'$ obtained by applying a Hurwitz move to $H_1\ldots H_s$, is a lift of the
product $k_1'\ldots k_s'$ obtained by applying the same Hurwitz move to 
the product $k_1\ldots k_s$. Let $M$ now be one of the matrices in cases $wI_n-IV$, 
and let $G_1\ldots G_r$ be
a factorization of $M$ in terms of conjugates of $U$. Let $m=g_1\ldots g_r$  where $m=\pi(M)$ and 
$g_i=\pi(G_i)$ for each $i=1,\ldots,r$. Since each $G_i$ is a conjugate of $U$, then each 
$g_i$ is a conjugate of $\pi(U)$. Also by definition the product $G_1\ldots G_r$ is a lift
of the product $g_1\ldots g_r$. The claim guarantees the existence of a finite sequence of 
Hurwitz moves which transforms $g_1\ldots g_r$ into a new product $(g_1'\ldots g_n')(g_{n+1}'\ldots
g_r')$. It follows that if $G_1'\ldots G_r'$ is the product obtained from $G_1\ldots G_r$
by applying the same sequence of Hurwitz moves, then $G_1'\ldots G_r'$ is a lift of $(g_1'\ldots g_n')(g_{n+1}'\ldots
g_r')$ where each $G_i$ is a conjugate of $U$. Now, by observing that
$U$ and $-U$ have different traces and therefore do not belong to the same 
conjugacy class, we conclude that $G_1'\ldots G_n'$ has to be the m.n.f. of $M$, and therefore
that $G_{n+1}'\ldots G_r'=Id_{2\times 2}$.

Cases $I^{\ast}_n-IV^{\ast}$ can be dealt with as follows. Let $M$ be one of the
matrices in cases $I^{\ast}_n-IV^{\ast}$. Suppose that $G_1\ldots G_r$ is a factorization of
$M$ in terms of conjugates of $U$.  By applying the homomorphism $\pi$ we obtain a factorization
$g_1\ldots g_r$ of $\pi(M)$ in terms of conjugates of $u=\pi(U)$. Since $\pi(M)=\pi(-M)$
and $-M$ is one of the matrices in cases $wI_n-IV$, we can apply the claim to $\pi(M)=g_1\ldots g_r$.
We conclude that there exists a sequence of Hurwitz moves which transforms the product
$g_1\ldots g_r$ into a product $(g_1'\ldots g_n')(g_{n+1}'\ldots g_r')$ where the m.n.f. of 
$-M$ is a lift of $g_1'\ldots g_n'$ and $g_{n+1}'\ldots g_r'=\pi(Id_{2 \times 2})$. Let 
$(G_1'\ldots G_n')(G_{n+1}'\ldots G_r')$ be the product obtained from $G_1\ldots G_r$ by applying the
same sequence of Hurwitz moves. By the observations made immediately after the claim we know that 
the facts that $G_1'\ldots G_n'$ is a lift of $g_1'\ldots g_n'$ and that each $G_i'$ is a conjugate of $U$ imply
that $G_1'\ldots G_n'$ has to be the m.n.f. of $-M$, and therefore that $G_{n+1}'\ldots G_r'=-Id_{2 \times 2}$
since $M=G_1'\ldots G_r'=(-M)(G_{n+1}'\ldots G_r')$. This finishes the proof of the
theorem.

In order to prove the claim we may first apply the automorphism 
$\rho$ (defined in Remark \ref{automorfismo}) and then prove the equivalent
claim for $\rho(m)$. Notice that after doing this the image of the
canonical factorization of $M$ becomes a factorization of $\rho(m)$ in
terms of the elements $\rho(u)=s_{1}$ and $\rho(v)=s_{0}$.
Now we prove the claim by analyzing each of the four possible cases.

Case 1: $m=u^{n}$ hence $\rho (m)=\rho (u)^{n}=s_{1}^{n}$. It suffices to prove that
for each $n\geq 0$, if $s_1^n=g_1\ldots g_r$ where each $g_i$ is a conjugate of
$s_1$, then $r\geq n$ and there exists a sequence of Hurwitz moves which transforms
the product $g_1\ldots g_r$ into a product $(g_1'\ldots g_n')(g_{n+1}'\ldots g_r')$ where
the m.n.f. of $M$ is a lift of $g_1'\ldots g_n'$ and $g_{n+1}'\ldots g_r'=\pi(Id_{2 \times 2})$.
In the case $n=0$ the m.n.f. of $Id_{2\times 2}$ is taken to an empty product.
We proceed by induction on $n$. The result is immediate when $n=0$. 
Let us suppose
that $s_{1}^{n}=g_{1}\cdots g_{r}$. By Proposition \ref{Key}, after applying a 
finite number of Hurwitz moves one arrives at a new product $s_{1}^{n}=g_{1}'\cdots g_{r}'$ 
in which either any pair of
consecutive $g_{i}$'s in this product join well or all factors are
short conjugates of $s_{1}$. 
In
the first case, by Lemma \ref{left} we know that the reduced expression of
this product must be of the form $\mathrm{left}(g_{1}')t_{1}\cdots t_{l}$. On
the other hand, $g_{1}'$ cannot be a long conjugate $q^{-1}s_{1}q$.
This is because the reduced expression of $s_{1}^{n}$ is $
b(wb^{2})^{n-1}wb$, the reduced expression of $\mathrm{left}
(q^{-1}s_{1}q)=q^{-1}s_{1}$ has the form $l_{1}\cdots l_{s}bwb$ and the
sequence $bwb$ does not appear in the reduced expression of $s_{1}^{n}$. In a
similar way, $g_{1}'$ cannot be $s_{0}$ or $s_{2}$ since $\mathrm{left}
(s_{0})=w$ and $\mathrm{left}(s_{2})=b^{2}$ but the reduced expression of $%
s_{1}^{n}$ starts with the element $b$. Hence $g_{1}'=s_{1},$ and we can
cancel out this element on both sides of $s_1^n=g_1'\ldots g_r'$ and apply the induction hypothesis to
obtain the result.

In the second case, i.e. when all the $g_{i}'$'s are short, we may
assume that there is at least one pair of consecutive elements that do not
join well, for otherwise we would be in the previous case. By Proposition \ref%
{short}, after a finite number of Hurwitz moves one arrives at a product $g_1''\ldots g_r''$ with 
$g_{1}''=s_{1}$. Canceling out this element in equation $s_1^n=g_1''\ldots g_r''$ and applying the induction hypothesis 
one obtains the result.

Case 2: $m=vu$, hence $\rho (m)=s_{0}s_{1}$. Let us suppose that $%
s_{0}s_{1}=g_{1}\cdots g_{r}$. Again, by Proposition \ref{Key}, after
a finite number of Hurwitz moves we arrive at a new product $s_0s_1=g_1'\ldots g_r'$ (e.1) in which either any pair of
consecutive $g_{i}'$'s join well or all factors are short conjugates of $s_{1}$. In
the first case, since $s_{0}s_{1}=b$, $g_{1}'$ can neither be long nor
equal to $s_{0}$ or $s_{2}$, for exactly the same reason as in the
previous case. We conclude that $g_{1}'=s_{1}$. Since $s_{0}s_{1}=s_{1}s_{2}$, we can
cancel out $s_{1}$ on both sides of (e.1) in order to obtain $%
s_{2}=g_{2}'\cdots g_{r}'$. Again, by Lemma \ref{left} we know that the
reduced expression of this product must be of the form $\mathrm{left}%
(g_{2}')t_{1}\cdots t_{k}$ and this must be equal to $\mathrm{left}%
(s_{2})=b^{2}$. This rules out the possibility of $g_{2}'$ being a long
conjugate or equal to $s_{1}$ or $s_{0}$. Thus, $g_{2}'=s_{2}$ and after applying a finite sequence of 
Hurwitz moves (e.1) can
be written in the form $s_{0}s_{1}=s_{1}s_{2}g_{3}''\cdots
g_{r}''$. As in the proof of
Proposition \ref{short}, one extra Hurwitz move allows us to write 
$s_{1}s_{2}$ as $s_{0}s_{1}$ and the claim follows.

In the second case, i.e. when all the $g_{i}'$'s are short we may assume that at 
least one pair of consecutive
elements does not join well. By Proposition \ref{short} after applying a finite sequence of 
Hurwitz moves one arrives at a product $s_0s_1=g_1''\ldots g_r''$ (e.2) with
$g_{1}''=s_{0}$. Canceling out $s_0$
on both sides of (e.2) we obtain $s_{1}=g_{2}''\cdots g_{r}''$. But
Case 1 (with $n=1$) implies that this product can be changed using Hurwitz
moves into a new one $g_2'''\ldots g_r'''$ with $g_{2}'''=s_{1}$. The claim follows.

The strategy of the proof for the remaining cases will be the same. In order
to make it shorter, we will abbreviate by Case $\ast$i, the case where all pairs
of consecutive elements in a product join well, and by Case $\ast$ii, the case where all elements 
in a product are short and at least one pair of consecutive elements does not join
well. We will also abbreviate the expression \emph{after a finite number of
Hurwitz moves} by \emph{after H.m.}.

Case 3: $m=vuv$, hence $\rho (m)=s_{0}s_{1}s_{0}$. Suppose that 
$s_{0}s_{1}s_{0}=g_{1}\cdots g_{r}$. After H.m. we arrive at the product $s_0s_1s_0=g_1'\ldots g_r'$ (e.3)
falling in one of the following cases.

Case 3i: Since $s_{0}s_{1}s_{0}=bwb^{2}$, the comparison of the
left sides of the reduced expressions of both sides of (e.3) implies that $g_{1}'=s_1$. Since
$s_{0}s_{1}s_{0}=s_{1}s_{2}s_{0}$ we have $s_{1}s_{2}s_{0}=g_1'\ldots g_r'$ and we can cancel out $s_{1}$ in this equation to obtain 
$s_{2}s_{0}=g_{2}'\cdots g_{r}'$. Now, $s_{2}s_{0}=s_{0}s_{1}$ implies that 
$s_{0}s_{1}=g_{2}'\cdots g_{r}'$ and we are in the previous case. We know that after H.m. the product $g_{2}'\cdots g_{r}'$
can be changed to a new one of the form 
$s_{0}s_{1}g_{4}''\cdots g_{r}''$ with $g_{4}''\ldots
g_{r}''=\pi(Id_{2\times 2})$, and consequently 
$s_{0}s_{1}s_{0}=s_{1}s_{0}s_{1}g_{4}''\cdots g_{r}''$. After H.m. the right hand side can be transformed into 
$s_{1}s_{2}s_{0}g_{4}''\cdots g_{r}''$ and this into 
$s_{0}s_{1}s_{0}g_{4}''\cdots g_{r}''$. The claim follows.

Case 3ii: Proposition \ref{short} allows us to assume that
$g_{r}'=s_0$ and if we cancel it out in (e.3) we obtain 
$s_{0}s_{1}=g_{1}'\cdots g_{r-1}'$ and ending up again in the previous case. After
H.m. the product $g_{1}'\cdots g_{r-1}'$ transforms into a new product of the form $s_0s_1g_3''\ldots g_{r-1}''$ 
with $g_{3}''\cdots$ $g_{r-1}''=\pi(Id_{2\times 2})$. Thus $s_{0}s_{1}s_{0}=s_{0}s_{1}g_{3}''\cdots
g_{r-1}''s_{0}$ which after H.m. becomes $s_{0}s_{1}(gs_{0}g^{-1})g_{3}''\cdots g_{r-1}''$ where $g=g_{3}''\cdots g_{r-1}=\pi(Id_{2\times 2})$.  The latter product is therefore $s_{0}s_{1}s_{0}g_{3}''\cdots g_{r-1}''$ and the claim follows.

Case 4: $m=(vu)^{2}$ hence $\rho (m)=(s_{0}s_{1})^{2}.$ Suppose 
$(s_{0}s_{1})^{2}=g_{1}\ldots g_{r}$. After H.m. we arrive at the product $s_0s_1s_0=g_1'\ldots g_r'$ (e.4)
falling in one of the following cases.

Case 4i: Since $(s_{0}s_{1})^{2}=b^{2}$, the comparison of the left
sides of the reduced expressions of both sides of (e.4) implies that $g_{1}'=s_2$. Since $(s_{0}s_{1})^{2}=s_{2}s_{0}s_{0}s_{1}$ we have that $s_{2}s_{0}s_{0}s_{1}=g_1'\ldots g_r'$.
Cancelling out $s_2$ on both sides of this equation gives 
$s_{0}s_{0}s_{1}=g_{2}'\ldots g_{r}'$ (e.5). Now $s_{0}s_{0}s_{1}=w$ and therefore $\text{left}(g_2')=w$.   
This implies that $g_{2}'=s_{0}$ and by cancelling this term out in equation (e.5) we obtain
$s_{0}s_{1}=g_{3}'\ldots g_{r}'$. Since this is Case 2, we know that after H.m. the product $g_{3}'\ldots g_{r}'$ transforms
into a product of the form $s_0s_1g_5''\ldots g_r''$ and we have 
$(s_{0}s_{1})^{2}=s_{2}s_{0}s_{0}s_{1}g_{5}''\ldots g_{r}''$ with $g_{5}''\ldots g_{r}''=\pi(Id_{2 \times 2})$. 
But $s_{2}s_{0}s_{0}s_{1}$ can be changed
with one extra Hurwitz move into $s_{0}s_{1}s_{0}s_{1}$ and the claim
follows.

Case 4ii: in this case, by Proposition \ref{short}, after H.m. the product $g_1'\ldots g_r'$ transforms into
a product product $(s_0s_1)^2=g_1''\ldots g_r''$ (e.6) with $g_r''=s_1$. Cancelling this element in equation
(e.6) gives $s_{0}s_{1}s_{0}=g_{1}''\cdots g_{r-1}''$. Since this is Case 3, we know that after H.m. the product
$g_{1}''\cdots g_{r-1}''$ transforms into a product of the form $s_0s_1s_0g_4'''\ldots g_{r-1}'''$ with $g_4'''\ldots g_{r-1}'''=\pi(Id_{2 \times 2})$. We have therefore that $s_0s_1s_0s_1=s_0s_1s_0g_4'''\ldots g_{r-1}'''s_1$, and after
H.m. this product transform into the product $s_0s_1s_0(gs_1g^{-1})g_4'''\ldots g_{r-1}'''$ with $g=g_4'''\ldots g_{r-1}'''=\pi(Id_{2 \times 2})$. The claim follows.

\end{proof}

\section{Confluence of singular fibers in elliptic fibrations}

In this section we apply Proposition \ref{proposition} and Theorem \ref{Main} 
to the question of giving necessary and sufficient conditions under which the 
set of singular fibers in an elliptic fibration can be fused 
into a unique singular fiber.
 
\begin{theorem}
\label{coalescencia2}Let $\phi:S\rightarrow D$ be a relatively minimal singular
elliptic fibration 
without multiple fibers. 
Then $\phi$ is weakly deformation equivalent to the elliptic fibration 
$\phi_T:S_T\rightarrow
D,$ if and only if the total monodromy 
$\lambda_{r,q_0,j}([C_{r}])$ of $\phi$ is 
conjugate with the matrix $M_T$ in Table \ref{primeratabla}
and $\chi(S)$ equals $\chi(S_T)$.
\end{theorem}

\begin{proof}
Necessity. Let us assume that $\phi:S\rightarrow D$ is weakly deformation equivalent to a 
$\phi_T:S_T\rightarrow D$ as described in the statement. Remark \ref{nota1} implies
that the total monodromies $\lambda_{r,q_0,j}([C_r])$ of $\phi$, and $\lambda_{r',q_0',j'}([C_{r'}])$ of $\phi_T$
are conjugate of each other, but the latter is conjugate with $M_T$. Remark \ref{nota2} implies that 
$\chi(S)=\chi(S_T)$.

Sufficiency.
By Theorem \ref{morsificacion} there exist morsifications $(\Phi,\mathcal{S},D\times D_{\epsilon})$ and $(\Psi,\mathcal{T},D\times D_{\delta})$ of $\phi:S\rightarrow D$ and $\phi_T:S_T\rightarrow D$, respectively. According
to that theorem we may assume that none of the members of either morsification contains a multiple fiber.
Let us fix $t\neq 0$ in $D_{\epsilon}$ and $t'\neq 0$ in $D_{\delta}$ and let us consider the elliptic 
fibrations $(\Phi_{t},\mathcal{S}_{t},D)$ and $(\Psi_{t'},\mathcal{T}_{t'},D)$. We claim that these 
elliptic fibrations are topologically equivalent. Let us denote by $q_1,\ldots q_k$ (resp. $q_1',\ldots,q_l'$) the
critical values of $\Phi_{t}$ (resp. $\Psi_{t'}$).
We begin to prove the claim by pointing out that $k$ must be equal to $l$, since
 $k=\sum_{i=1}^{k}\chi(\Phi_t^{-1}(q_i))=\chi(\mathcal{S}_t
)=\chi(S)=\chi(S_T)=\chi(\mathcal{T}_{t'})=\sum_{i=1}^{l}\chi(\Psi_t^{-1}(q_i'))=l$ where the first and last 
equalities are justified by the fact that all singular fibers of $\Phi_{t}$ and $\Psi_{t'}$ are of type $I_1$ and therefore each one of these fibers has Euler characteristic $1$, and the third and fifth equalities are justified by
Remark \ref{nota1}. Since $\chi(S_T)$ always equals the number of factors $n_T$ in the m.n.f. of $M_T$, we conclude that $k=n_T$.
Let $\lambda=\lambda_{r,q_0,j}$ and $\lambda'=\lambda_{r',q_0',j'}$ be monodromy representations
of $\Phi_t$ and $\Psi_{t'}$, respectively, and let $\gamma_1,\ldots,\gamma_k$ and $\gamma_1',\ldots,\gamma_k'$ be special bases for $\pi_1(D-\{q_1,\ldots,q_{n_T}\},q_0)$ and $\pi_1(D-\{q_1',\ldots,q_{n_T}'\},q_0')$, respectively. Now, 
the total mondromy $\lambda([C_r])$ of $\Phi_t$ is conjugate of a total monodromy of $\phi$ and the total monodromy
$\lambda'([C_{r'}])$ of $\Psi_{t'}$ is conjugate of a total monodromy of $\phi_T$, we have that 
$\lambda([C_r])$ and $\lambda'([C_{r'}])$ are conjugates of each other, say $\lambda([C_r])=A^{-1}\lambda'([C_{r'}])A$ for some $A \in SL(2,\mathbb{Z})$. On the other hand, 
$\lambda([C_r])=\lambda(\gamma_1)\ldots \lambda(\gamma_{n_T})$ and $\lambda'([C_{r'}])=\lambda'(\gamma_1')\ldots \lambda'(\gamma_{n_T}')$
are factorizations of the corresponding total monodromies in terms of conjugates of $U$. It is clear that there exist $B,C\in SL(2,\mathbb{Z})$
such that $M_T=(B^{-1}\lambda(\gamma_1)B)\ldots (B^{-1}\lambda(\gamma_{n_T})B)=(C^{-1}\lambda'(\gamma_1')C)\ldots (C^{-1}\lambda'(\gamma_{n_T}')C)$.  By Theorem \ref{Main} the product $(B^{-1}\lambda(\gamma_1)B)\ldots (B^{-1}\lambda(\gamma_{n_T})B)$ can be transformed into the product $(C^{-1}\lambda'(\gamma_1')C)\ldots (C^{-1}\lambda'(\gamma_{n_T}')C)$ by applying performing a finite number of Hurwitz moves. By using the immediate fact
that if in a group a product $g_1\ldots g_r$ can be transformed by applying Hurwitz moves to another product
$h_1\ldots h_r$, then for any element $g\in G$, the product $(g^{-1}g_1g)\ldots(g^{-1}g_rg)$ can be transformed
into the product $(g^{-1}h_1g)\ldots(g^{-1}h_rg)$ by applying Hurwitz moves, we obtain that $\lambda(\gamma_1)\ldots \lambda(\gamma_{n_T})$ can be transformed into the product $(BC^{-1}\lambda'(\gamma_1')(BC^{-1})^{-1})\ldots (BC^{-1}\lambda'(\gamma_{n_T}')(BC^{-1})^{-1}))$. We conclude that the product $\lambda(\gamma_1)\ldots \lambda(\gamma_{n_T})$ transforms into the product $\lambda'(\gamma_1')\ldots \lambda'(\gamma_{n_T}')$ by a finite sequence of Hurwitz moves followed by a conjugation of all factors in the product by the same element is $SL(2,\mathbb{Z}$.  But Proposition \ref{proposition} implies that under these circumstances, the elliptic fibrations
$\Phi_t$ and $\Psi_{t'}$ are topologically equivalent. In conclusion, $(\Phi,\mathcal{S},D\times D_{\epsilon})$ and
$(\Psi,\mathcal{T},D\times D_{\delta})$ are deformations of $(\phi,S,D)$ and $(\phi_M,S_M,D)$, respectively, such that
$(\Phi_t,\mathcal{S}_t,D)$ and $(\Psi_{t'},\mathcal{T}_{t'},D)$ are topologically equivalent, and therefore $(\phi,S,D)$ and $(\phi_M,S_M,D)$ are weakly deformation equivalent. 
\end{proof}

\end{document}